\documentclass[12pt]{article}
\usepackage{amsmath,amssymb,amsfonts,amsthm}
\newenvironment{myabstract}{\par\noindent
{\bf Abstract . } \small }
{\par\vskip8pt minus3pt\rm}
\newcounter{item}[section]
\newcounter{kirshr}
\newcounter{kirsha}
\newcounter{kirshb}
\newenvironment{enumroman}{\setcounter{kirshr}{1}
\begin{list}{(\roman{kirshr})}{\usecounter{kirshr}} }{\end{list}}
\newenvironment{enumarab}{\setcounter{kirshb}{1}
\begin{list}{(\arabic{kirshb})}{\usecounter{kirshb}} }{\end{list}}
\newenvironment{athm}[1]{\vskip3mm\par\noindent
{\bf #1 }. \slshape }
{\upshape\par\vskip10pt minus3pt}
\newtheorem{theorem}{Theorem}[section]

\newtheorem{lemma}[theorem]{Lemma}

\newenvironment{demo}[1]{\noindent{\bf #1.}\upshape\mdseries}
{\nopagebreak{\hfill\rule{2mm}{2mm}\nopagebreak}\par\normalfont}
\theoremstyle{definition}

\newtheorem{example}[theorem]{Example}
\newtheorem{definition}[theorem]{Definition}

\def\C{{\mathfrak{C}}}
\def\Fm{{\mathfrak{Fm}}}

\def\Nr{{\mathfrak{Nr}}}

\def\Sg{{\mathfrak{Sg}}}
\def\Fm{{\mathfrak{Fm}}}
\def\A{{\mathfrak{A}}}
\def\B{{\mathfrak{B}}}
\def\C{{\mathfrak{C}}}
\def\D{{\mathfrak{D}}}

\def\Sn{{\mathfrak{Sn}}}

\def\Lf{{\bf Lf}}
\def\Dc{{\bf Dc}}

\def\K{{\bf K}}
\def\K{{\bf K}}

\def\RCA{{\bf RCA}}

\def\(R)RA{{\bf (R)RA}}

\def\Dc{{\bf Dc}}

\def\Dc{{\bf Dc}}

\def\Di{{\bf Di}}

\def\B{{\sf B}}
\def\G{{\sf G}}

\def\K{{\sf K}}

\def\Nr{{\mathfrak{Nr}}}

\def\Nr{{\mathfrak{Nr}}}

\def\Zd{{\mathfrak Zd}}
\def\A{{\mathfrak{A}}}
\def\B{{\mathfrak{B}}}
\def\C{{\mathfrak{C}}}
\def\D{{\mathfrak{D}}}

\def\A{{\mathfrak{A}}}
\def\B{{\mathfrak{B}}}
\def\C{{\mathfrak{C}}}
\def\D{{\mathfrak{D}}}

\def\L{{\mathfrak{L}}}

\def\L{{\mathfrak{L}}}
\def\Bl{{\mathfrak{Bl}}}

\def\RCA{{\bf RCA}}

\def\G{{\bf G}}

\def\Fl{{\mathfrak{Fl}}}

\def\Nr{{\mathfrak{Nr}}}

\def\RCA{{\bf RCA}}

\def\Sg{{\mathfrak Sg}}
\def\P{{\mathfrak P}}
\def\Rl{{\mathfrak Rl}}

\def\Ig{{\sf Ig}}

\title{The amalgamation in Boolean, and non-Boolean algebras with operators}
\author{Tarek Sayed Ahmed \\
Department of Mathematics, Faculty of Science,\\ 
Cairo University, Giza, Egypt.
  }
\begin{document}
\maketitle

\begin{myabstract} We study the connections of the global amalgamation property, and some of its variations,
in the context of classes (not necessarily varieties)
of boolean algebras with operators, to the local interpolation and the congruence extension property 
on the free algebrs of the varieties generated by such classes.
\end{myabstract}

\section{Introduction}

The amalgamation property
(for classes of models), since its discovery, has played a dominent role in algebra and model theory. 
Algebraic logic is the natural interface between universal algebra and logic (in our present context a variant of first order logic). 
Indeed, in algebraic logic amalgamation properties in classes of algebras are proved to be equivalent to interpolation 
results in the corresponding 
logic. Pigozzi and Comer  worked  such equivalences for cylindric algebras, the latter for finite dimensions the former for infinite ones.

The principal context of \cite{P} is the class of {\it infinite dimensional} cylindric 
algebras, an equational formalism of first order logic.
In this paper Pigozzi deals basically with the following question:
Which subclasses of infinite dimensional cylindric algebras, 
other than  the class of locally finite ones, 
still have the (strong) amalgamation property.

The fact that the class of locally finite cylindric algebras 
has the strong amalgamation property, proved earlier by Diagneualt  
is equivalent to the fact that first order logic has the Craig interpolation property. 
The classes that Pigozzi deals with consist solely of algebras that are 
infinite dimensional and we assume, to simplify notation, 
that such classes of algebras are $\omega$-dimensional, where
$\omega$ is the least infinite ordinal. 
These classes include the class of $\omega$-dimensional 
locally finite algebras $(\Lf_{\omega}$), 
the class of dimension complemented algebras ($\Dc_{\omega}),$  
the class of $\omega$-dimensional diagonal algebras 
(exact definitions will be recalled below), and 
the class of $\omega$ dimensional 
semisimple algebras. 
Here a semisimple algebra is a subdirect product
of simple algebras.
All of these classes consist exclusively of algebras that are
representable, but  unlike $\RCA_{\omega}$ 
none of these classes is first order definable, least a variety. 
While the amalgamation property speaks about amalgamating algebras 
in such a way that the amalgam only agrees
on the common subalgebra, the amalgamation is said to be {\it strong}
if the common subalgebra 
is the {\it only} overlap between the two algebras in the amalgam.
The positive results of section 2.2 in combination
with the negative ones of section 2.3 of \cite{P} answer 
most of the natural questions
one could ask about amalgamation for cylindric algebras.  
In particular, Pigozzi proves that in the (strictly) increasing sequence
$$\Lf_{\omega}\subset \Dc_{\omega}\subset \Di_{\omega}\subset \RCA_{\omega}$$ 
the first and third classes have the amalgamation property while 
the second and fourth fail
to have it. However, most questions concerning the strong amalgamation property
for several classes of cylindric algebras were posed as open questions,  and
other closely related ones appeared after Pigozzi's paper was published.
In \cite{AUU} all of Pigozzi's questions are answered.

Here we carry out similar investigations in a much broader context, that of Boolean algebras with operators.
As a by product of our investigations we obtain several new results concerning algebraisations of first order logic, other than 
cylindric algebras. We will also have occasion to weaken the Boolean structure, 
dealing with non-classical or many valued logics.

\section{ Amalgamation}
We star by the relevant definitions:
\begin{definition}
\begin{enumarab}
\item $K$ has the \emph{Amalgamation Property } if for all $\A_1, \A_2\in K$ and monomorphisms
$i_1:\A_0\to \A_1,$ $i_2:\A_0\to \A_2$
there exist $\D\in K$
and monomorphisms $m_1:\A_1\to \D$ and $m_2:\A_2\to \D$ such that $m_1\circ i_1=m_2\circ i_2$.
\item If in addition, $(\forall x\in A_j)(\forall y\in A_k)
(m_j(x)\leq m_k(y)\implies (\exists z\in A_0)(x\leq i_j(z)\land i_k(z) \leq y))$
where $\{j,k\}=\{1,2\}$, then we say that $K$ has the superamalgamation property $(SUPAP)$.
\end{enumarab}
\end{definition}

\begin{definition} An algebra $\A$ has the {\it strong interpolation theorem}, $SIP$ for short, if for all $X_1, X_2\subseteq A$, $a\in \Sg^{\A}X_1$,
$c\in \Sg^{\A}X_2$ with $a\leq c$, there exist $b\in \Sg^{\A}(X_1\cap X_2)$ such that $a\leq b\leq c$.
\end{definition}

For an algebra $\A$, $Co\A$ denotes the set of congruences on $\A$.
\begin{definition}

An algebra $\A$ has the {\it congruence extension property}, or $CP$ for short,
 if for any $X_1, X_2\subset A$
if $R\in Co \Sg^{\A}X_1$ and $S\in Co \Sg^{A}X_2$ and
$$R\cap {}^2\Sg^{A}(X_1\cap X_2)=S\cap {}^2\Sg^{\A}{(X_1\cap X_2)},$$
then there exists a congruence $T$ on $\A$ such that
$$T\cap {}^2 \Sg^{\A}X_1=R \text { and } T\cap {}^2\Sg^{\A}{(X_2)}=S.$$

\end{definition}

Theorems \ref{supapgeneral}, \ref{CP}, \ref{weak} to come, 
give a flavour of the interconnections between the local properties of $CP$ and $SIP$ (on free algebras) and the global property of 
superamalgamation (of the entire class). 
Maksimova and Mad\'arasz proved that if interpolation holds in free algebras of a variety, then the variety has the
superamalgamation property.
Using a similar argument, we prove this implication in a slightly more
general setting. But first an easy  lemma:

\begin{lemma} Let $K$ be a class of $BAO$'s. Let $\A, \B\in K$ with $\B\subseteq \A$. Let $M$ be an ideal of $\B$. We then have:
\begin{enumarab}
\item $\Ig^{\A}M=\{x\in A: x\leq b \text { for some $b\in M$}\}$
\item $M=\Ig^{\A}M\cap \B$
\item if $\C\subseteq \A$ and $N$ is an ideal of $\C$, then
$\Ig^{\A}(M\cup N)=\{x\in A: x\leq b+c\ \text { for some $b\in M$ and $c\in N$}\}$
\item For every ideal $N$ of $\A$ such that $N\cap B\subseteq M$, there is an ideal $N'$ in $\A$
such that $N\subseteq N'$ and $N'\cap B=M$. Furthermore, if $M$ is a maximal ideal of $\B$, then $N'$ can be taken to be a maximal ideal of $\A$.

\end{enumarab}
\end{lemma}
\begin{demo}{Proof} Only (iv) deserves attention. The special case when $n=\{0\}$ is straightforward.
The general case follows from this one, by considering
$\A/N$, $\B/(N\cap \B)$ and $M/(N\cap \B)$, in place of $\A$, $\B$ and $M$ respectively.
\end{demo}
The previous lemma will be frequently used without being explicitly mentioned.

\begin{theorem}\label{supapgeneral} Let $K$ be a class of $BAO$'s such that $\mathbf{H}K=\mathbf{S}K=K$.
Assume that for all $\A, \B, \C\in K$, inclusions
$m:\C\to \A$, $n:\C\to \B$, there exist $\D$ with $SIP$ and $h:\D\to \C$, $h_1:\D\to \A$, $h_2:\D\to \B$
such that for $x\in h^{-1}(\C)$,
$$h_1(x)=m\circ h(x)=n\circ h(x)=h_2(x).$$
Then $K$ has $SUPAP$.
\bigskip
\bigskip
\bigskip
\bigskip
\bigskip
\bigskip
\bigskip
\bigskip
\bigskip
\bigskip
\bigskip
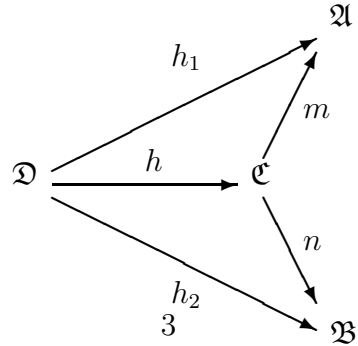
\begin{picture}(10,0)(-30,70)
\thicklines
\put (-10,0){$\D$}
\put(5,0){\vector(1,0){70}}\put(80,0){$\C$}
\put(5,5){\vector(2,1){100}}\put(110,60){$\A$}
\put(5,-5){\vector(2,-1){100}}\put(110,-60){$\B$}
\put(85,10){\vector(1,2){20}}
\put(85,-5){\vector(1,-2){20}}
\put(40,5){$h$}
\put(100,25){$m$}
\put(100,-25){$n$}
\put(50,45){$h_1$}
\put(50,-45){$h_2$}
\end{picture}
\end{theorem}
\bigskip
\bigskip
\begin{proof} Let $\D_1=h_1^{-1}(\A)$ and $\D_2=h_2^{-1}(\B)$. Then $h_1:\D_1\to \A$, and $h_2:\D_2\to \B$.

Let $M=ker h_1$ and $N=ker h_2$, and let
$\bar{h_1}:\D_1/M\to \A, \bar{h_2}:\D_2/N\to \B$ be the induced isomorphisms.

Let $l_1:h^{-1}(\C)/h^{-1}(\C)\cap M\to \C$ be defined via $\bar{x}\to h(x)$, and
$l_2:h^{-1}(\C)/h^{-1}(\C)\cap N$ to $\C$ be defined via $\bar{x}\to h(x)$.
Then those are well defined, and hence
$k^{-1}(\C)\cap M=h^{-1}(\C)\cap N$.
Then we show that $\P=\Ig(M\cup N)$ is a proper ideal and $\D/\P$ is the desired algebra.
Now let $x\in \mathfrak{Ig}(M\cup N)\cap \D_1$.
Then there exist $b\in M$ and $c\in N$ such that $x\leq b+c$. Thus $x-b\leq c$.
But $x-b\in \D_1$ and $c\in \D_2$, it follows that there exists an interpolant
$d\in \D_1\cap \D_2$  such that $x-b\leq d\leq c$. We have $d\in N$
therefore $d\in M$, and since $x\leq d+b$, therefore $x\in M$.
It follows that
$\mathfrak{Ig}(M\cup N)\cap \D_1=M$
and similarly
$\mathfrak{Ig}(M\cup N)\cap \D_2=N$.
In particular $P=\mathfrak{Ig}(M\cup N)$ is a proper ideal.

Let $k:\D_1/M\to \D/P$ be defined by $k(a/M)=a/P$
and $h:\D_2/N\to \D/P$ by $h(a/N)=a/P$. Then
$k\circ m$ and $h\circ n$ are one to one and
$k\circ m \circ f=h\circ n\circ g$.
We now prove that $\D/P$ is actually a
superamalgam. i.e we prove that $K$ has the superamalgamation
property. Assume that $k\circ m(a)\leq h\circ n(b)$. There exists
$x\in \D_1$ such that $x/P=k(m(a))$ and $m(a)=x/M$. Also there
exists $z\in \D_2$ such that $z/P=h(n(b))$ and $n(b)=z/N$. Now
$x/P\leq z/P$ hence $x-z\in P$. Therefore  there is an $r\in M$ and
an $s\in N$ such that $x-r\leq z+s$. Now $x-r\in \D_1$ and $z+s\in\D_2,$ it follows that there is an interpolant
$u\in \D_1\cap \D_2$ such that $x-r\leq u\leq z+s$. Let $t\in \C$ such that $m\circ
f(t)=u/M$ and $n\circ g(t)=u/N.$ We have  $x/P\leq u/P\leq z/P$. Now
$m(f(t))=u/M\geq x/M=m(a).$ Thus $f(t)\geq a$. Similarly
$n(g(t))=u/N\leq z/N=n(b)$, hence $g(t)\leq b$. By total symmetry,
we are done.
\end{proof}

The intimate relationship between $CP$ on free algebras generating a certain variety and the $AP$ for such varieties, 
has been worked out extensively by Pigozzi
for various classes of cylindric algebras. Here we prove an implication in one direction for $BAO$'s.
Notice that we do not assume that our class is a variety.
\begin{theorem}\label{CP}
Let $K$ be such that $\mathbf{H}K=\mathbf{S}K=K$. If $K$ has the amalgamation  property, then the $V(K)$
free algebras, on any set of generators, have $CP$.
\end{theorem}

\begin{proof}
For $R\in Co\A$ and $X\subseteq  A$, by $(\A/R)^{(X)}$ we understand the subalgebra of
$\A/R$ generated by $\{x/R: x\in X\}.$ Let $\A$, $X_1$, $X_2$, $R$ and $S$ be as specified in in the definition of $CP$.
Define $$\theta: \Sg^{\A}(X_1\cap X_2)\to \Sg^{\A}(X_1)/R$$
by $$a\mapsto a/R.$$
Then $ker\theta=R\cap {}^2\Sg^{\A}(X_1\cap X_2)$ and $Im\theta=(\Sg^{\A}(X_1)/R)^{(X_1\cap X_2)}$.
It follows that $$\bar{\theta}:\Sg^{\A}(X_1\cap X_2)/R\cap {}^2\Sg^{\A}(X_1\cap X_2)\to (\Sg^{\A}(X_1)/R)^{(X_1\cap X_2)}$$
defined by
$$a/R\cap {}^{2}\Sg^{\A}(X_1\cap X_2)\mapsto a/R$$
is a well defined isomorphism.
Similarly
$$\bar{\psi}:\Sg^{\A}(X_1\cap X_2)/S\cap {}^2\Sg^{\A}(X_1\cap X_2)\to (\Sg^{\A}(X_2)/S)^{(X_1\cap X_2)}$$
defined by
$$a/S\cap {}^{2}\Sg^{\A}(X_1\cap X_2)\mapsto a/S$$
is also a well defined isomorphism.
But $$R\cap {}^2\Sg^{\A}(X_1\cap X_2)=S\cap {}^2\Sg^{\A}(X_1\cap X_2),$$
Hence
$$\phi: (\Sg^{\A}(X_1)/R)^{(X_1\cap X_2)}\to (\Sg^{\A}(X_2)/S)^{(X_1\cap X_2)}$$
defined by
$$a/R\mapsto a/S$$
is a well defined isomorphism.
Now
$(\Sg^{\A}(X_1)/R)^{(X_1\cap X_2)}$ embeds into $\Sg^{\A}(X_1)/R$ via the inclusion map; it also embeds in $\A^{(X_2)}/S$ via $i\circ \phi$ where $i$
is also the inclusion map.
For brevity let $\A_0=(\Sg^{\A}(X_1)/R)^{(X_1\cap X_2)}$, $\A_1=\Sg^{\A}(X_1)/R$ and $\A_2=\Sg^{\A}(X_2)/S$ and $j=i\circ \phi$.
Then $\A_0$ embeds in $\A_1$ and $\A_2$ via $i$ and $j$ respectively.
Then there exists $\B\in V$ and monomorphisms $f$ and $g$ from $\A_1$ and $\A_2$ respectively to
$\B$ such that
$f\circ i=g\circ j$.
Let $$\bar{f}:\Sg^{\A}(X_1)\to \B$$ be defined by $$a\mapsto f(a/R)$$ and $$\bar{g}:\Sg^{\A}(X_2)\to \B$$
be defined by $$a\mapsto g(a/R).$$
Let $\B'$ be the algebra generated by $Imf\cup Im g$.
Then $\bar{f}\cup \bar{g}\upharpoonright X_1\cup X_2\to \B'$ is a function since $\bar{f}$ and $\bar{g}$ coincide on $X_1\cap X_2$.
By freeness of $\A$, there exists $h:\A\to \B'$ such that $h\upharpoonright_{X_1\cup X_2}=\bar{f}\cup \bar{g}$.
Let $T=kerh $. Then it is not hard to check that
$$T\cap {}^2 \Sg^{\A}(X_1)=R \text { and } T\cap {}^2\Sg^{\A}(X_2)=S.$$
\end{proof}
Finally we show that $CP$ implies a weak form of interpolation.
\begin{theorem}\label{weak}
If an algebra $\A$ has $CP,$  then for $X_1, X_2\subseteq \A$, if $x\in \Sg^{\A}X_1$ and $z\in \Sg^{\A}X_2$ are such that
$x\leq z$, then there exists $y\in \Sg^{\A}(X_1\cap X_2),$ and a term $\tau$ such that $x\leq y\leq \tau(z)$.
If $Ig^{\Bl\A}\{z\}=\Ig^{\A}\{z\},$ then $\tau$ can be chosen to be the
identity term. In particular, if $z$ is closed, or $\A$ comes from a discriminator variety,
then the latter case occurs.
\end{theorem}

\begin{proof}
Now let $x\in \Sg^{\A}(X_1)$, $z\in \Sg^{\A}(X_2)$ and assume that $x\leq z$.
Then $$x\in (\Ig^{\A}\{z\})\cap \Sg^{\A}(X_1).$$
Let $$M=\Ig^{\A^{(X_1)}}\{z\}\text { and } N=\Ig^{\Sg^{\A}(X_2)}(M\cap \Sg^{\A}(X_1\cap X_2)).$$
Then $$M\cap \Sg^{\A}(X_1\cap X_2)=N\cap \Sg^{\A}(X_1\cap X_2).$$
By identifying ideals with congruences, and using the congruence extension property,
there is a an ideal $P$ of $\A$
such that $$P\cap \Sg^{\A}(X_1)=N\text { and }P\cap \Sg^{\A}(X_2)=M.$$
It follows that
$$\Ig^{\A}(N\cup M)\cap \Sg^{\A}(X_1)\subseteq P\cap \Sg^{\A}(X_1)=N.$$
Hence
$$(\Ig^{(\A)}\{z\})\cap A^{(X_1)}\subseteq N.$$
and we have
$$x\in \Ig^{\Sg^{\A}X_1}[\Ig^{\Sg^{\A}(X_2)}\{z\}\cap \Sg^{\A}(X_1\cap X_2).]$$
This implies that there is an element $y$ such that
$$x\leq y\in \Sg^{\A}(X_1\cap X_2)$$
and $y\in \Ig^{Sg^{\A}X}\{z\}$, hence the first required. The second required follows
follows, also immediately, since $y\leq z$, because $\Ig^{\A}\{z\}=\Rl_z\A$.
\end{proof}
 
We note that all of the above results hold for $MV$ algebras which satisfy all axioms of Boolean algebras except idempotency.

\section{Sheaf theoretic duality and epimorphisms}

Here we deal with non-classical logics; we review some known basic notions and concepts culminating in defining the algebras we shall deal with.
Our work closely follows Comer, except that we deal with Zarski topologies rather than Stone topologies. 
We obtain an analogous representabilty theorem to the effect that evey theory can be represented as the continous sections of a 
Sheaf. We start with the origin of our algebras.

\begin{definition}A $t$ norm is a binary operation $*$ on $[0,1]$, i.e $(t:[0,1]^2\to [0,1]$) such that
\begin{enumroman}
\item  $*$ is commutative and associative,
that is for all $x,y,z\in [0,1]$,
$$x*y=y*x$$
$$(x*y)*z=x*(y*z).$$
\item $*$ is non decreasing in both arguments, that is
$$x_1\leq x_2\implies x_1*y\leq x_2*y$$
$$y_1\leq y_2\implies x*y_1\leq x*y_2.$$
\item $1*x=x$ and $0*x=0$ for all $x\in [0,1].$
\end{enumroman}
\end{definition}
The following are the most important (known) examples of continuous $t$ norms.

\begin{enumroman}
\item Lukasiewicz $t$ norm: $x*y=max(0,x+y-1)$
\item Godel $t$ norm $x*y=min(x,y)$
\item Product $t$ norm $x*y=x.y$
\end{enumroman}
We have the following known result \cite{H} lemma 2.1.6
 
\begin{theorem} Let $*$ be a continuous $t$ norm. 
Then there is a unique operation $x\implies y$ satisfying for all $x,y,z\in [0,1]$, the condition $(x*z)\leq y$ iff $z\leq (x\implies y)$, namely 
$x\implies y=max\{z: x*z\leq y\}$
\end{theorem}
The operation $x\implies y$ is called the residuam of the $t$ norm. The residuam $\implies$ 
defines its corresponding unary operation of precomplement 
$(-)x=(x\implies 0)$.
The Godel negation satisfies $(-)0=1$, $(-)x=0$ for $x>0$.
Abstracting away from $t$ norms, we get

\begin{definition} A residuated lattice is an algebra
$$(L,\cup,\cap, *, \implies 0,1)$$ 
with four binary operations and two constants such that
\begin{enumroman}
\item $(L,\cup,\cap, 0,1)$ is a lattice with largest element $1$ and the least element $0$ (with respect to the lattice ordering defined the usual way: 
$a\leq b$ iff $a\cap b=a$).
\item $(L,*,1)$ is a commutative semigroup with largest element $1$, that is $*$ is commutative, associative, $1*x=x$ for all $x$.
\item Letting $\leq$ denote the usual lattice ordering, we have $*$ and $\implies $ form an adjoint pair, i.e for all $x,y,z$
$$z\leq (x\implies y)\Longleftrightarrow x*z\leq y.$$
\end{enumroman}
\end{definition}
$BL$ algebras, introduced and studied by Hajek \cite{H}, are what is called $MTL$ algebras satisfying the identity $x*(x\implies y)=x\cap y$.
Both are residuated lattices with extra conditions. The propositional logic $MTL$ was introduced by
Esteva and Godo \cite{E}.  It has three basic connectives $\to$, $\land$ and $\&$.
We say that $\L$ is a core fuzzy logic if $\L$ expands $MTL$, $\L$ has the Local Deduction Theorem $(LDT)$, 
and $\L$ satisfies 
(*) $\phi\equiv \psi\vdash \chi(\phi)\equiv \chi(\psi)$ for all formulas $\phi,\psi,\chi$. (Here $\equiv$ is defined via $\&$ and $\implies$).
The ($LDT)$ says that for a theory $T$ and a formula 
$\phi$, whenever  $T\cup \{\phi\}\vdash \psi$, then there exists a natural number $n$ such that $T\vdash \phi^n\to \psi$.
Here $\phi^n$ is defined inductively by $\phi^1=\phi$ and $\phi^n=\phi^{n-1}\&\phi$.
Thus core fuzzy logics are axiomatic expansions of $MTL$ having $LDT$ and obeying the substitution rule (*).
The basic notions of evaluation, tautology and model  for core fuzzy logics  are defined the usual way. 
Let $\L$ be a core fuzzy logic and $I$ the set of additional connectives of $\L$. An $\L$ algebra is a structure 
$\B=(B, \cup, \cap, *, \implies, (c_B)_{c\in I},0,1)$
such that $(B,\cup, \cap, *, \implies, 0,1)$ is an $MTL$ algebra and each additional axiom of $\L$ is a tautology of $\B$. 
Throughout the paper the operations of algebras are denoted by $\cup$, $\cap$, $\implies$ 
$*$ and the corresponding logical operations by $\lor, \land, \to, \&$.

Generalizing a very nice result of Comer we represent $BAO$;s 
as the continous sections of sheaves, 
the representation here is indeed a functor that is strongly invertble.
We start from a concrete example adressing varaints and extension first order logics. 
The following discussion applies to $L_n$, $L_{\omega,\omega}$, $Dc$, Keislers logics with and without equality, finitray logics of infinitary
relations. It also applies to non classical ologics, whose Stone space is the Zarski topology.

\begin{example}

Let $\L$ be a logic,  and $T$ be a theory in $\Fm_L$.
Let $\Sn_L$ denote the set of sentences, that is formulas with no 
free variables.  We assume that $T\in \Sn_L$ has no free variable, let $X_T=\{\Delta\subseteq \Sn_L: \Delta \text{ is complete }\}$.
This is simply the stone space of $\Sn_T$. For each $\Delta\in X_T$ let $\Fm_{\Delta}$
be the corresponding Tarski-Lindenbaum algebra.
Let $\delta T$ be the following disjoint union
$$\bigcup_{\Delta\in X_T}\{\Delta\}\times \Fm_{\Delta}.$$ Define the following topolgies, on $X_T$ and $\delta T$, respectively.

On $X_{\Gamma}$ the Stone topology, and on $\delta_{\Gamma}$ 
the topology with base $B_{\psi,\phi}=\{\Delta, [\phi]_{\Delta}, \psi\in \Delta, \Delta\in \Delta_{\Gamma}\}.$
Then $(X_T, \delta T)$ is a {\it sheaf}, and its dual consisting of continous sections, 
$\Gamma(T,\Delta)\cong \Fm_T$. 

Then the contnious sections of the sheaves 
$\Gamma(X_T, \delta(T))\cong \Fm_T$.
\end{example}

\begin{example}
It also applies to non classical logic. Let $\L$ be a predicte language for $BL$ algebras (This for example incudes $MV$ algebras).
Let $X_T$ be the Zarski topology on $\Sn$ based on $\{\Delta\in Max: a\notin \Delta\}$. 
Let $\delta T=\bigcup _{\Delta\in X_T}\{\Delta\}\times \Fm_{\Delta}$. 
\end{example}

\begin{definition} Let $\B$ be an algebra. A filter of $\B$ is a nonempty subset $F\subseteq A$ such that for all $a,b\in B$,
\begin{enumroman}
\item $a,b\in F$ implies $a*b\in F.$
\item $a\in F$ and $a\leq b$ imply $b\in F.$
\end{enumroman}
\end{definition}
It easy to check that if $F$ is a filter on $A$ then $1\in F$ and whenever $a, a\implies b\in F$ then $b\in F$.
Also $a*b\in F$ if and only if $a\cap b\in F$ iff $a\in F$ and $b\in F$. A filter $F$ is proper if $F\neq A$ and it is easy to see that a filter $F$ is proper
iff $0\notin F$. 

\begin{definition} A filter $P$ of $A$ is prime provided that it is a prime filter of the underlying lattice $L(\B)$ of $\B$, that is
$a\cup b\in P$ implies $a\in P$ or $b\in P$. This is equivalent to the statement that for all $a,b\in \B,$ $a\implies b\in P$ or $b\implies a\in P$.
A proper filter $F$ is maximal if it is not properly contained in any other proper filter. 
\end{definition}
We let $Max(\B)$ denote the set of maximal filters
and $Spec(\B)$ the family of prime filters. Then it is not hard to actually show that $Max(\B)\subseteq Spec(\B)$ \cite{spec}.
For a set $X\subseteq \B$, $\Fl^{\B}X$ denotes the filter generated by $X$. 
A filter $F$ is called principal, if $F=\Fl\{a\}=\{x\in B: x\geq a\}$.
The following notions are taken from \cite{spec}. Proofs are also found in \cite{spec}.
Let $\B$ be a non-trivial algebra. For each $X\subseteq \B$, we set
$$V(X)=\{P\in Spec(X): X\subseteq P\}.$$ Then the family $\{V(X)\}_{X\subseteq \B}$ of subsets of $spec(\B)$ 
satisfies the axioms for closed sets in a topological space. The resulting topology is called the 
Zariski topology, and the resulting topological space is called the prime spectrum of $\B$. We write $V(a)$ for the more cumbersome $V(\{a\})$.
For any $X\subseteq \B$, let
$$D(X)=\{P\in Spec(X): X\nsubseteq P\}$$
Then $\{D(X)\}_{X\subseteq A}$ is the family of open sets of the Zariski topology.
We write $D(a)$ for $D(\{a\})$.
The minimal spectrum of $\B$ is the topology induced by the Zariski topology on $Max(\B)$.
For $X\subseteq \B$ and $a\in \B$, let
$$V_M(X)=V(X)\cap Max(\B)\text { and } D_M(X)=D(X)\cap Max(\B).$$
$$V_M(a)=V(a)\cap Max(\B),\text { and } D_M(a)=D(a)\cap Max(\B).$$
In other words,
$$V_M(a)=\{F\in Max(\B): a\in F\}$$ 
and $$D_M(a)=\{F\in Max(\B): a\notin F\}.$$

\begin{lemma} Let $\B$ be an algebra. Let $a,b\in \B$.
Then the following hold:
\begin{enumroman}
\item $D_M(a)\cap D_{M}(b)=D_M(a\cup b).$
\item  $D_M(a)\cup D_M(b)=D_M(a\cap b)=D_M(a*b).$
\item $D_M(X)=Max(\B)$  iff $\Fl^{\B}X=\B.$
\item  $D_M(\bigcup_{i\in I}X_i)=\bigcup_{i\in I}D_M(X_i).$
\item $V_M(a)\cap V_M(b)=V_M(a\cap b).$
\item $a\leq b$ if and only if $V_M(a)\subseteq V_M(b).$
\end{enumroman}
\end{lemma}
\begin{demo}{Proof} \cite{spec} proposition 2.8. We only prove one side of the last item, since it is not mentioned in \cite{spec}.
Assume that $V_a\subseteq V_b$. If it is not the case that $a\leq b$, then we may assume that $a\cap (b\implies 0)$ is not $0$.
Hence there is a proper maximal filter $F$, such that $a\cap (b\implies  0)\in F.$ Hence $a\in F$ and $b\to 0$ is in $F$.
But this implies that $b\notin F$ lest $0\in F$. Hence $F\in V_a$ and $F\notin V_b.$ This is a contradiction, and the required is proved.
\end{demo}
\begin{lemma} If $F$ is a maximal filter in a $BL$ algebra $\A$, then for any $a\in A$ 
either $a$ or $a\to 0$ is in $F.$
\end{lemma}
\begin{demo}{Proof}Let $\A\in BL$. Assume that both $a$ and $-a=a\to 0$ are not in $F$. Then, by maximality,  the filter generated by
$F\cup \{-a\}$ is the whole algebra $\A$. Then $a\geq x\cap -a$, for some $x\in F$. 
Hence $0=a\cap x\cap -a=x\cap -a$. But then $x\leq a$ and $a\in F$ after all.
\end{demo}
\begin{theorem}\label{d}Let $\B$ be an algebra. 
\begin{enumroman}
\item $\{D_M(a)\}_{a\in \B}$ is a basis for a compact Hausdorff topology on $Max(\B)$
\item Furthermore if $a=\bigvee a_i$, then $V_M(a)\sim \bigcup V_M(a_i)$ is a nowhere dense subset of $Max(\B)$. 
Similarly if $a=\bigwedge a_i$, then $\bigcap V_M(a_i)\sim V_M(a).$
is nowhere dense. 
\item If $\B$ is countable, then $Max(\B)$ is a Polish space.
\end{enumroman}
\end{theorem} 
\begin{demo}{Proof} 
\begin{enumroman}
\item We include the proof for self completeness and also because the `nowhere density' part is completely new, 
and as we shall see in a while it will play a pivotal role in the proof of the omitting types theorem.
That $Max(\B)$ is compact and Hausdorff  is proved in \cite{spec}, theorem 2.9, the proof goes as follows:
Assume that $$Max(\B)=\bigcup_{i\in I}D_M(a_i)=D_M(\bigcup_{i\in I}a_i).$$ Then $\B=\Fl\{\bigcup_{i\in I} a_i\}$, hence $0\in \Fl\{\bigcup_{i\in I}a_i\}$.
There is an $n\geq 1$ and $i_1,\ldots i_n\in I$ such that $a_{i_1}*\ldots a_{i_n}=0$. But 
$$Max(\B)=D_{M}(0)=D_M(a_{i_1}*\ldots a_{i_n})=D_M(a_{i_1})\cup\ldots D_M(a_{i_n}).$$
Hence every cover is reducible to a finite subcover. Hence the space is compact.
Now we show that it is Hausdorff. Let $M$, $N$ be distinct maximal filters. 
Let $x\in M\sim N$ and $y\in N\sim M$. Let $a=x\implies  y$ and $b=y\implies  x$. Then $a\notin M$ and $b\notin N$. Hence 
$M\in D_M(a)$ and $N\in D_M(b)$. Also $D_M(a)\cap D_M(b)=D_M(a\lor b)=D_M(1)=\emptyset$.
 We have proved that the space is Hausdorff.

\item Now assume that $a=\bigvee a_i$ and $V_M(a)\sim \bigcup V_M(a_i)$ is not nowhere dense. 
Then there exists $d$ such that $D_M(d)\subseteq V_M(a)\sim V_M(a_i)$
Hence $$V_M(a_i)\subseteq V_M(a)\sim D_M(d)=V_M(a)\cap V_M(d)=V_{M}(a\cap d).$$ It follows that
$a\cap d=a$ so $a\leq d$. Then $D_M(d)\subseteq D_M(a)$. So we have, 
$D_M(d)\subseteq D_M(a)\cap V_M(a)=\emptyset$ contradiction.
Conversely assume that $a=\bigwedge{a_i}$ 
and assume that $$D_M(d)\subseteq \bigcap V_M(a_i)\sim V_M(a).$$
Let $e=d\to 0$. Then $V_M(e)=D_M(d)$.
Now we have
$$V_M(e)\subseteq \bigcap V_M(a_i)\sim V_M(a).$$
Taking complements twice, we get
$$V_M(e)\subseteq D_M(a)\sim \bigcup D_M(a_i)$$
Then
$V_M(e)\subseteq D_M(a)\sim D_M(a_i)$. So $$D_M(a_i)\subseteq D_M(a)\sim V_M(e)=D_M(a)\cap D_M(e)=D_M(a\cup e).$$
Hence $V_M(a\cup e)\subseteq V_M(a_i)$. So $a\cup e\leq a_i$ for each $i$. Thus $a\cup e=a$ from 
which we get that $e\leq a$. Hence $V_M(e)\subseteq V_M(a)$.
But $V_M(e)\subseteq D_M(a)$ it follows that $V_M(e)=\emptyset$. But $V_M(e)=D_M(d)$ and we are done.
\item If $\B$ is countable, then $Max{\B}$ is second countable, so the required follows from (i) together with theorem \ref{p}.
\end{enumroman}
\end{demo}

We consider a class $\K$ of $BL$ algebras with operators $(BLO$s).
If $\A\in \K_{\alpha}$ and $X\subseteq A$, then $\Ig^{\A}X$ denotes the ideal generated by $X$.
For $x\in A$, we define $\Delta x=\{i\in I: f_ix\neq x\}$. We assume that 
$\Delta x=\Delta(-x)$, $\Delta (x\cap y)\subseteq \Delta x\cap \Delta y$.

$\Zd\A$ denotes the Boolean algebra 
That is $Zd\A=\{x\in \A: f_ix=x,\  \forall i\in \alpha\}$. 
If $\A$ is a locally finite cylindric algebra of formulas, then $\Zd\A$ is the Boolean 
algebra of sentences.

We describe a functor that associates to each $BLO$ a pair of topolgical spaces
space $(X(\A), \delta(\A))=\A^d$, where $\delta(\A)$ has an algebraic structure, as well; in fact it is a subdircet product
of algebras, that are simple under favourable circumstances, in which case $\delta(\A)$ is sa semisimple algebra carring 
a product topology. This pair is  called the dual space of $\A$.

$X(\A)$ is the Zarski topology of  $\Zd\A$, defined on the prime spectrum. 

Now we turn to defining the second component; this is more involved. 
For $x\in X(\A)$, let $\G_x=\A/\Ig^{\A}x$ (the stalk over $x$) and 
$$\delta(\A)=\bigcup\{\G_x: x\in X(\A)\}.$$
This is clearly a disjoint union, and hence 
it can also be regarded as the following product $\prod_{x\in \A} \G_x$ of algebras. This is not semi-simple, because $x$ is only maximal in $\Zd\A$. 
But the semisimple case will deserve special attention.

The projection $\pi:\delta(\A)\to X(\A)$ is defined for $s\in \G_x$ by $\pi(s)=x$. For $a\in A$, 
we define a function
$\sigma_a: X(\A)\to \delta(\A)$ by $\sigma_a(x)=a/\Ig^{\A}x\in \G_x$. 

Now we define the topology on 
$\delta(\A)$. It is the smallest topology  for which all these functions are open, so $\delta(\A)$ 
has both an algebraic structure and a topological one, and they are compatible.

We can turn the glass around. Having such a space we associate an algebra in $\K$.
Let $\pi:\G\to X$ denote the projection associated with the space $(X,\G)$, built on $\A$.
A function  $\sigma:X\to \G$ is a section of $(X,\G)$ if $\pi\circ \sigma$ is the identity 
on $X$. 

Dually, the construction of the corresponding $BAO$ from a reduced space, uses the sectional functor.
The set $\Gamma(X,\G)$ of all continuous sections of 
$(X,\G)$ becomes a $BAO$ by defining the operations pointwise, recall that $\G=\prod \G_x$ is a $BLO$.
The mapping $\eta:\A\to \Gamma(X(\A), \delta(\A))$ defined by $\eta(a)=\sigma_a$ 
is as easily checked  an isomorphism, completing the invertibility of the functor. 

Note that under this map an element in $Zd\A$ corresponds with the characteristic 
function $\sigma_N\in \Gamma(X, \delta)$ 
of the clopen set $N_a$.

Given two spaces $(Y,\G)$ and $(X,\L)$ a sheaf morphism $H:(Y,\G)\to (X,\L)$ is a pair $(\lambda,\mu)$ where $\lambda:Y\to X$ is a continous map
and $\mu$ is a continous map $Y+_{\lambda} \L\to \G$ such that $\mu_y=\mu(y,-)$ is a homomorphism of $\L_{\lambda(y)}$ into $\G_y$.
We consider $Y+_{\lambda} \L=\{(y,t)\in Y\times \L:\lambda(y)=\pi(t)\}$ as a subspace of $Y\times \L$.
That is, it inherits its topology from the product topology on $Y\times \L$.

A sheaf morphism $(\lambda,\mu)=H:(Y,\G)\to (X,\L)$ produces a $BAO$ homomorphism 
$\Gamma(H):\Gamma(X,\L)\to \Gamma(Y,\G)$ the natural way:
for $\sigma\in \Gamma(X,\L)$ define $\Gamma(H)\sigma$ by $(\Gamma(H)\sigma)(y)=\mu(y, \sigma(\lambda y))$ for all $y\in Y$.
A sheaf morphism $h^d:\B^d\to \A^d$ can also be asociated with a homomorphism $h:\A\to \B$. 
Define $h^d=(h^*, h^o)$ where for $y\in X(\B)$, $h^*(y)=h^{-1}\cap Zd\A$ and for $y\in X(\B)$ and $a\in A$
$$h^0(h, a/\Ig^{\A}h^*(y))=h(a)/\Ig^{\B}y.$$

\begin{athm}{Definition} An algebra $\A$  is nice if whenever $x$ is a prime ideal in $Zd\A$, then $\Ig^{\A}x$ is a maximal ideal in
$\A$.
\end{athm}
It is easy to see that locally finite algebras are nice. 
For a class of algebras $K$ we say that $K$ has $ES$ if epimorphisms (in the categorial sense) are surjective.
We will show that $ES$ fails in the class of simple algebras defined above, some are cylindric-like, other are not. 
\begin{athm}{Theorem} Let $V$ be a class of algebras, such that the simple algebras in $V$ have the amalgamation property.
Assume that there exist nice algebras $\A,\B\in V$ and an epimorphism $f:\A\to \B$ that is not onto.
Then $ES$ fails in the class of simple algebras.
\end{athm}

\begin{demo}{Proof} Suppose, to the contrary that $ES$ holds for simple algebras.
Let $f^*:\A\to \B$ be the given epimorphism that is not onto. We assume that $\A^d=(X,\L)$ and $\B^d=(Y,\G)$ 
are the corresponding dual sheaves over the Boolean spaces $X$ and $Y$ and by  duality that 
$(h,k)=H:(Y,\G)\to (X,\L)$ is a monomorphism. Recall that $X$ is the set of maximal ideals in $Zd\A$, and similarly for $Y$.
We shall first prove
\begin{enumroman}
\item $h$ is one to one
\item for each $y$ a maximal ideal in $\Zd\B$, $k(y,-)$ is a surjection of the stalk over $h(y)$ onto the stalk over $y$.
\end{enumroman}
Suppose that $h(x)=h(y)$ for some $x,y\in Y$. Then $\G_x$, $\G_y$ and $\L_{hx}$ are simple algebra, 
so there exists a simple $\D\in V$ and monomorphism $f_x:\G_x\to \D$ and $f_y:\G_y\to \D$ such that
$$f_x\circ k_x=f_y\circ k_y.$$
Here we are using that the algebras considered are nice, and that the simple algebras have $AP$.
Consider the sheaf $(1,D)$ over the one point space $\{0\}=1$ and sheaf morphisms 
$H_x:(\lambda_x,\mu):(1,D)\to (Y,\G)$ and $H_y=(\lambda_y, v):(1,D)\to (Y,\G)$
where $\lambda_x(0)=x$ $\lambda_y(0)=y $ $\mu_0=f_x$ and $v_0=f_y$. The sheaf $(1,\D)$ is the space dual to $\D\in V$
and we have $H\circ H_x=H\circ H_y$. Since $H$ is a monomorphism $H_x=H_y$ that is $x=y$.
We have shown that $h$ is one to one.
Fix $x\in Y$. Since, we are assuming that  $ES$ holds for simple algebras of $V,$ in order to show that 
$k_x:\L_{hx}\to \G_x$ is onto, it suffices to show that $k_x$ is an epimorphism.  
Hence suppose that $f_0:\G_x\to \D$ and $f_1:\G_x\to \D$ for some simple $\D$ such that $f_0\circ k_x=f_1\circ k_x$.
Introduce sheaf morphisms 
$H_0:(\lambda,\mu):(1,\D)\to (Y,\G)$ and $H_1=(\lambda,v):(1,\D)\to (Y,\G)$
where $\lambda(0)=x$, $\mu_0=f_0$ and $v_0=f_1$. Then $H\circ H_0=H\circ H_1$, 
but $H$ is a monomorphism, so we have $H_0=H_1$ from which 
we infer that $f_0=f_1$. 

We now show that (i) and (ii) implies that $f^*$ is onto, which is a contradiction.
Let $\A^d=(X,\L)$ and $\B^d=(Y,\G)$. It suffices to show that $\Gamma((f^*)^d)$ is onto (Here we are taking a double dual) . 
So suppose $\sigma\in \Gamma(Y,\G)$. For each $x\in Y$, 
$k(x,-)$ is onto so $k(x,t)=\sigma(x)$ for some $t\in \L_{h(x)}$. That is $t=\tau_x(h(x))$ for some 
$\tau_x\in \Gamma(X,\G)$. Hence there is a clopen neighborhood $N_x$ of $x$ such that 
$\Gamma(f^*)^d)(\tau_x)(y)=\sigma(y)$ for all $y\in N_x$. 
Since $h$ is one to one and $X,Y$ are Boolean spaces, we get that $h(N_x)$ is clopen in $h(Y)$ and there is a 
clopen set $M_x$ in $X$ such that $h(N_x)=M_x\cap h(Y)$. Using compactness, there exists a partition of
$X$ into clopen subsets $M_0\ldots M_{k-1}$ and sections $\tau_i\in \Gamma(M_i,L)$ such that
$$k(y,\tau_i(h(y))=\sigma(y)$$ 
wherever $h(x)\in M_i$ for $i<k$. Defining 
$\tau$ by $\tau(z)=\tau_i(z)$ whenever $z\in M_i$ $i<k$, it follows that $\tau\in \Gamma(X,\L)$ and $\Gamma((f^*)^d)\tau=\sigma$.
Thus $\Gamma((f^*)^d)$ is onto $\Gamma(\B^d)$, and we are done. 
\end{demo}

\section{Logical application}

\subsection{Beth definability}

Here by algebra, we mean either cylindric, Pinter, quasipolyadic, or quasipolyadic equality algebra.
The next theorem, whose proof wil be omitted, will help us obtain two new results.

\begin{lemma}
\begin{enumarab}
\item Let $\A$ be semisimple simple, then there exists a unique 
$\delta(\A)\in \Dc_{\alpha+\omega}$ 
$i:\A\to \Nr_{\alpha}\delta(\A)$. The algebra $\delta(\A)$ is called an $\omega$ dilation of $\A$ 
Futhermore, if $\A\cong \B$, then this isomorphism lifts to $\delta(\A)\cong \delta(\B)$. 
\item  Semisimple algebras have $AP$ with respect to the representables
\item Simple algebras have $AP$.
\end{enumarab}
\end{lemma} 

In an unpublished manuscript of the author two nice $BL$ algebras with opeartors were constructed 
such that the inclusion is an epimorphism that is not surjective
For {\it all} cylindric-algebras, infinite dimensional  nice algebras as in the statement of the theorem were constructed by Madarazs. 
We readily obtain that in all these varieties epimorphisms are not surjective even in simple algebras, because 
by the last  theorem simple algebras have $AP$.

\end{document}